\documentclass[12pt]{amsart}

\usepackage{amssymb}
\usepackage{amscd}
\usepackage{statex}
\usepackage{enumerate}

\title[An elementary abelian $p$-cover of the Hermitian curve]{An elementary abelian $p$-cover of the Hermitian curve with many automorphisms}
\author{Herivelto Borges \and Satoru Fukasawa}


\address{Universidade de S\~ao Paulo, Inst. de Ci\^encias Matem\'aticas
 e de Computa\c c\~ao, S\~ao Carlos, SP 13560-970,  Brazil.} 
\email{hborges@icmc.usp.br}

\address{Department of Mathematical Sciences, Faculty of Science, Yamagata University, Kojirakawa-machi 1-4-12, Yamagata 990-8560, Japan.}
\email{s.fukasawa@sci.kj.yamagata-u.ac.jp}
\thanks{The second author was partially supported by JSPS KAKENHI Grant Number JP19K03438}

\newtheorem{thm}{Theorem}
\newtheorem{prop}{Proposition}

\newtheorem{lem}{Lemma}

\theoremstyle{definition}
\newtheorem{rem}{Remark}

\begin{document}
\begin{abstract}
The full automorphism group of a certain elementary abelian $p$-cover of the Hermitian curve in characteristic $p>0$ is determined. 
It is remarkable that the order of Sylow $p$-groups of the automorphism group is close to Nakajima's bound in terms of the $p$-rank. 
Weierstrass points, Galois points, Frobenius nonclassicality, and  arc property are also investigated. 
\end{abstract}
\maketitle

\section{Introduction} 
Let $k$ be an algebraically closed field of characteristic $p \ge 0$. The automorphism group ${\rm Aut}(X)$ of a smooth projective curve $X$ of genus $g_X \ge 2$ over $k$ has been studied by many authors. 
When $p=0$, there is a well-known bound of $|{\rm Aut}(X)|$ by Hurwitz: $|{\rm Aut}(X)| \le 84(g_X-1)$.  
When $p>0$, there are  many examples of curves $X$ for which  $|{\rm Aut}(X)| >84(g_X-1)$. 
In many of these  cases, the $p$-rank $\gamma_X$ of $X$ is zero. 
It is known that $\gamma_X \le g_X$ \cite{hasse-witt}, and  curves $X$  for which $\gamma_X=g_X$ and $|{\rm Aut}(X)|>84(g_X-1)$ were first presented by Subrao \cite{subrao}. 
The relation between $\gamma_X$ and ${\rm Aut}(X)$ was investigated by Nakajima \cite{nakajima}. 
In positive characteristic, a significant problem in the study of algebraic curves  with  large automorphism groups is that of constructing  a family of curves $X$ such that $\gamma_X>0$ and $|{\rm Aut}(X)|/g_X \rightarrow \infty$ as $g_X  \rightarrow \infty$  (see \cite{giulietti-korchmaros1, giulietti-korchmaros2}).

Hereafter, let us  assume that $p>0$ and $q$ is a power of $p$. 
This article considers  the smooth model $X_n$ of the plane curve $C_n \subset \mathbb{P}^2$ of degree $q^{2n}(q+1)$ of affine equation
$$ \left(\sum_{i=0}^n\alpha_i x^{q^{2i}}\right)^{q+1}+\left(\sum_{i=0}^n\alpha_i y^{q^{2i}}\right)^{q+1}+c=0, $$
where $\alpha_n=1$, $\alpha_0,c \in k \setminus \{0\}$, and $\alpha_1, \ldots, \alpha_{n-1} \in k$.  
It is proved  that $\gamma_{X_n}>0$ and $|{\rm Aut}(X_n)|/g_{X_n} \rightarrow \infty$ as $g_{X_n} \rightarrow \infty$, as a corollary of Theorem \ref{main1} (stated below). 
This curve belongs to a family of plane curves with two Galois points, which was introduced by the second author \cite{fukasawa2}. 
A point $P \in \mathbb{P}^2 \setminus C$ is called an outer Galois point for a plane curve $C \subset \mathbb{P}^2$ if the function field extension $k(C)/\pi_P^*k(\mathbb{P}^1)$ induced by the projection $\pi_P$ from $P$ is Galois (see \cite{miura-yoshihara, yoshihara}). 
This study  proves that  further Galois points for $C_n$ exist.
In addition, several important properties of $X_n$ are investigated, with  
the following two theorems  as the main results. 

\begin{thm} \label{main1} 
For the curve $X_n$, the following hold.  
\begin{itemize}
\item[(a)] The genus $g_{X_n}$ of $X_n$ is $\frac{q^{2n}(q+1)(q^{2n+1}-2)}{2}+1$. 
\item[(b)] The $p$-rank $\gamma_{X_n}$ of $X_n$ is $q^{4n+1}-q^{2n+1}-q^{2n}+1$. 
\item[(c)] ${\rm supp}(D)=\{ P \in X_n \ | \ H(P) \ni q^{2n+1}\}$, where $D$ is the divisor of $X_n$ arising from $C_n \cap \{Z=0\}$ and $H(P)$ is the Weierstrass semigroup of $P$.  
\item[(d)] $|{\rm Aut}(X_n)| = q^{4n+1}(q^2-1)(q+1)$. 
In more detail, there exists an exact sequence
$$ 1 \rightarrow \Sigma_0 \rtimes (\mathbb{Z}/(q+1) \mathbb{Z}) \rightarrow {\rm Aut}(X_n) \rightarrow PGL(2, \mathbb{F}_q) \rightarrow 1 $$
of groups, where $\Sigma_0=\{(x, y) \mapsto (x+\beta, y+\gamma) \ | \ \sum_i\alpha_i\beta^{q^{2i}}=\sum_i\alpha_i\gamma^{q^{2i}}=0\}$.  
\item[(e)] There exist exactly $q^2-q$ outer Galois points for $C_n$. 
\item[(f)] The group ${\rm Aut}(X_n)$ is generated by all automorphisms coming from outer Galois points.  
\end{itemize}
\end{thm}

\begin{thm} \label{main2}
Assume that $k=\overline{\mathbb{F}}_q$, and  let $s\geq 1$ be an integer. Then the curve  $C_n$ 
is $p^s$-Frobenius nonclassical if and only if $p^s=q^{2(n+1)}$ and  $C_n$  is  $\mathbb{F}_{q^{2(n+1)}}$-projectively equivalent to the curve
of affine equation
$$(x+x^{q^2}+\cdots+x^{q^{2n}})^{q+1}+(y+y^{q^2}+\cdots+y^{q^{2n}})^{q+1}+1=0.$$
 In the case of  Frobenius nonclassicality  the following hold.
\begin{itemize}
\item[(a)] $X_n$ has $q^{4n+3}-q^{4n+1}+q^{2n+1}+q^{2n}$ $\mathbb{F}_{q^{2(n+1)}}$-rational points. 
\item[(b)] The set $C_n(\mathbb{F}_{q^{2(n+1)}})$ forms a $(q^{4n+3}-q^{4n+1}+q+1, q^{2n+1}+q^{2n})$-arc, where $C_n(\mathbb{F}_{q^{2(n+1)}})$ is the set of all $\mathbb{F}_{q^{2(n+1)}}$-rational points of $\mathbb{P}^2$ lying on $C_n$.  
\item[(c)] The arc $C_n(\mathbb{F}_{q^{2(n+1)}})$ is not complete.
\end{itemize}
\end{thm}

\begin{rem}
\begin{itemize}
\item[(a)] According to Theorem \ref{main1} (a) and (d), it follows that $|{\rm Aut}(X_n)| \ge g_{X_n}^{\frac{2n+2}{2n+1}}$. 
In particular, for $n=1$, $|{\rm Aut}(X_1)| \ge g_{X_1}^{4/3}$.  
\item[(b)] Let $p>2$. 
It follows from Theorem \ref{main1} (d) that the order of Sylow $p$-groups of ${\rm Aut}(X_n)$ is equal to $q^{4n+1}$. 
According to Theorem \ref{main1} (b), this order $q^{4n+1}$ is close to the bound 
$$ |H| \le \frac{p}{p-2} (\gamma_X-1) $$
by Nakajima \cite[Theorem 1 (i)]{nakajima} on the order of Sylow $p$-groups $H$ of the automorphism group of curves $X$ with $\gamma_X \ge 2$.  
\item[(c)] It follows from Theorem \ref{main1} (c) that ${\rm Aut}(X_n)$ acts on ${\rm supp}(D)$, and that all points of ${\rm supp}(D)$ are Weierstrass points. 
\item[(d)] The number of outer Galois points for a plane curve $C$ is denoted by $\delta'(C)$. 
The studied  curves form a new family of curves $C$ of degree $d$ such that $\delta'(C) \rightarrow \infty$ as $d \rightarrow \infty$, by Theorem \ref{main1} (e) (see the Table for $\delta'(C)$ in \cite{yoshihara-fukasawa}).  
\item[(e)] Theorem \ref{main2} provides a family of curves giving a negative answer to a question raised in \cite[Question 1.1]{gptu}.
\end{itemize}
\end{rem}

\section{Proof of Theorem \ref{main1}}
The system of homogeneous coordinates on $\mathbb{P}^2$ is denoted by $(X:Y:Z)$, and the system of affine coordinates on $\mathbb{A}^2$ with $Z \ne 0$ is denoted by $(x, y)$, where $x=X/Z$ and $y=Y/Z$. 
Let 
$$ L(x)=\alpha_0 x+\alpha_1 x^{q^2}+\cdots+ \alpha_{n-1}x^{q^{2(n-1)}}+x^{q^{2n}}, $$
where $\alpha_0, \ldots, \alpha_{n-1} \in k$, with $\alpha_0\ne 0$, and let
$$ f(x, y)=L(x)^{q+1}+L(y)^{q+1}+c=0$$
be the defining equation of $C_n$. 
The irreducibility of $C_n$ was proved in \cite[Theorem 1]{fukasawa2}. 
The genus of $X_n$ was also computed in \cite{fukasawa2}. 
Assertion (a) in Theorem \ref{main1} follows. 

Let $\alpha^{q+1}=-1$ and let
$$ t=\alpha^{q} x+y,  \,  u=x+\alpha^{q} y. $$
Then 
$$ x=\frac{1}{\alpha^q+\alpha}(t+\alpha u), \ y=\frac{1}{\alpha^q+\alpha}(\alpha t+u), $$
and the relation between $t$ and $u$ is described as
$$ L(t)^qL(u)+L(t)L(u)^q+c'=0, $$
where $c'=(\alpha^q+\alpha)c$. 
This plane model is denoted by $C_n'$. 
Let $\psi: X_n \rightarrow C_n'$ be the normalization. 

\begin{rem}\label{rem2} For the singular points of $C_n$ and $C_n'$, the following hold.
\begin{itemize}
\item[(a)] ${\rm Sing}(C_n)=\{(\alpha :1: 0) \ | \ \alpha^{q+1}=-1\}$.  
\item[(b)] ${\rm Sing}(C_n')=\{(\delta:1:0) \ | \ \delta^q+\delta=0 \} \cup \{(1:0:0)\}$.  
\item[(c)] For a singular point $P=(\alpha:1:0) \in {\rm Sing}(C_n)$, the multiplicity of $P$ is $q^{2n}$, and there exist $q^{2n}$ tangent lines,  defined by $X-\alpha Y-\beta Z=0$ with $L(\beta)=0$. In particular, $P$ is an ordinary singular point.  Clearly, mutatis mutandis, the same holds for the points $P \in {\rm Sing}(C_n')$.
\end{itemize}
\end{rem}
 
For the plane model $C_n'$, there exist the following automorphisms. 
\begin{itemize}
\item[(a)] $\sigma_{\beta, \gamma}: (t, u) \mapsto (t+\beta, u+\gamma)$, where $L(\beta)=0$ and $L(\gamma)=0$. 
\item[(b)] $\tau_{\delta}: (t, u) \mapsto (t+\delta u, u)$, where $\delta^q+\delta=0$. 
\item[(c)] $(t, u) \mapsto (\epsilon t, \epsilon^{-q} u)$, where $\epsilon \in \mathbb{F}_{q^2}^*$. 
\end{itemize} 
Note that these automorphisms fix the singular point $(1:0:0)$ for the coordinates $(T: U: V)$ with $T/V=t$, $U/V=u$. 
For the plane model $C_n$, there exist the following automorphisms. 
\begin{itemize}
\item[(d)] $(x, y) \mapsto (\lambda x, y)$, where $\lambda^{q+1}=1$. 
\end{itemize}
Note that the automorphisms of type (d) act transitively on the set ${\rm Sing}(C_n)$. 
Let $\Sigma$ be the subgroup of ${\rm Aut}(X_n)$ generated by all automorphisms of type (a), (b), (c) and (d).  
The group $\Sigma$ acts transitively on ${\rm Sing}(C_n)$, and the order of any point-stabilizer subgroup is at least $(q^{2n})^2q(q^2-1)$. 
Thus $|{\rm Aut}(X_n)| \ge |\Sigma| \ge q^{4n+1}(q^2-1)(q+1)$, which is  part of Theorem \ref{main1}, assertion (d).

Let 
$$\Gamma:=\langle \sigma_{\beta, 0}, \tau_{\delta} \ | \ L(\beta)=0, \delta^q+\delta=0 \rangle \subset {\rm Aut}(X_n), $$ 
which is generated by automorphisms of type (a) with $\gamma=0$ and of type (b). 
Note that $k(X_n)^{\Gamma}=k(t, u)^{\Gamma}=k(u)$, and then $k(X_n)/k(u)$ is a Galois extension of degree $q^{2n+1}$, induced by the projection from the singular point $(1:0:0)$ of multiplicity $q^{2n}$. 
Since the projective line defined by a component of $L(u)$ intersects $C_n'$ at a unique point $(1:0:0)$, there exists a unique ramification point of index $q^{2n+1}$ for such a line. 
According to the Deuring--$\breve{\mbox{S}}$afarevi$\breve{\mbox{c}}$ formula (\cite{subrao}), 
$$ \frac{\gamma_{X_n}-1}{q^{2n+1}}=-1+q^{2n}\left(1-\frac{1}{q^{2n+1}}\right), $$
and  Theorem \ref{main1}, assertion (b) follows.

Let us  consider assertion (c) in Theorem \ref{main1} and prove that $H(P) \ni q^{2n+1}$ for any $P \in {\rm supp}(D)$. 
It  can be assumed that $\psi(P)=(1:0:0)$. 
As  shown by the proof of assertion (b) in Theorem \ref{main1}, for the automorphism group $\Gamma$, the covering map $X_n \rightarrow X_n/\Gamma \cong \mathbb{P}^1$ is totally ramified at $P$.   
This implies that $H(P) \ni q^{2n+1}$. 
The following lemma is needed for the proof of ${\rm supp}(D) \supset \{P \in X_n \ | \ H(P) \ni q^{2n+1}\}$.

\begin{lem}
The divisor $(q^{2n+1}-2)D$ is a canonical divisor. 
\end{lem}

\begin{proof}
Consider the projection of $C_n$ from $P=(1:0:0)$.  
Note that the Riemann--Hurwitz formula \cite[Remark 4.3.7]{stichtenoth} yields
$$(d y)=-2(y)_{\infty}+ q(L(x))_0.$$
Since  $(y)_{\infty}=D$ and $(L(x))=(L(x))_0-q^{2n}D$, it follows that $ (dy)\sim (q^{2n+1}-2)D.$ 
\end{proof}

The following proves that $q^{2n+1} \not\in H(Q)$ for any $Q=(a:b:1) \in C_n$.
Let us  assume that $x-a$ is a local parameter at $Q$ and    consider the canonical embedding induced by $|(q^{2n+1}-2)D|$. 
Note that $(x)_{\infty}=(y)_{\infty}=D$ and 
$$(x^{q+1}+y^{q+1})_\infty \le q D. $$
Since $(x-a)^{q+1}+(y-b)^{q+1}=\prod_{\alpha^{q+1}=-1}(x-a-\lambda(y-b))$, it follows that
$$ ((x-a)^{q+1}+(y-b)^{q+1})_\infty \le q D, $$
and then
$$(x-a)^{q^{2n+1}-q-2}\{(x-a)^{q+1}+(y-b)^{q+1}\} \in \mathcal{L}((q^{2n+1}-2)D).$$
Since the line defined by $X-a Z-\alpha(Y-b Z)=0$ with $\alpha^{q+1}+1=0$ intersects $C_n$ at a singular point $(\alpha:1:0)$, it follows that this line is not a tangent line at $Q$, and thus ${\rm ord}_Q(x-a-\alpha (y-b))=1$. 
Therefore,
$$ {\rm ord}_Q \left((x-a)^{q^{2n+1}-q-2}\{(x-a)^{q+1}+(y-b)^{q+1}\}\right)=q^{2n+1}-1,$$ 
and then  $q^{2n+1} \not\in H(Q)$. 
The proof of assertion (c) in Theorem \ref{main1} is completed.

To prove the equality $|{\rm Aut}(X_n)|=q^{4n+1}(q^2-1)(q+1)$ in assertion (d) of Theorem \ref{main1}, it will be shown that ${\rm Aut}(X_n) \subset PGL(3, k)$. 

\begin{prop} \label{completeness} 
The linear system induced by the normalization $\varphi: X_n \rightarrow C_n \subset \mathbb{P}^2$ is complete. 
\end{prop}

\begin{proof} 
Let us  take the plane model $C_n'$ with the system $(T:U:V)$ of homogeneous coordinates. 
Given $g \in \mathcal{L}(D) \setminus \{0\}$, consider the effective divisor  $D':={\rm div} (g)+D$.  
Set  $E:=(q^{2n}-1)D$ and write ${\rm div} (V^{q^{2n}})=q^{2n}D=D+E$. 
Note that the curve $V^{q^{2n}}=0$ is an adjoint of $C_n'$.
From Max-Noether's residue theorem (see \cite[Theorem 4.66, Theorem 6.51]{hkt}), it follows that there exists an adjoint $G(T,U,V)=0$ of $C_n'$ of degree $q^{2n}$ such that ${\rm div} (G(T,U,V))=D'+ E={\rm div} (g)+D+E$. 
Therefore, 
$$
{\rm div}(g(t, u))={\rm div}(G(T, U, V)/V^{q^{2n}}). 
$$
Without loss of generality, it can be  assumed that $\deg g(t, u) \le q^{2n}$. 
It is claimed that the polynomial $g(t, u)$ is of degree at most one. 
Note that points in $\psi^{-1}((1:0:0))$ are zeros of $L(u)$ and poles of $t$. 
Let $g=\sum_{i=0}^ma_i(u)t^i$ with $a_m(u) \ne 0$, where $0 \le m \le q^{2n}$. 
Assume that $m \ge 2$. 
Then $\deg a_m(u) \le q^{2n}-m \le q^{2n}-2$. 
Therefore, there exists a point $P \in \psi^{-1}((1:0:0))$ such that $a_m(P) \ne 0$, namely, ${\rm ord}_Pa_m(u)=0$. 
It follows that ${\rm ord}_P g(t, u)=-m \le -2$. 
This is a contradiction to $g \in \mathcal{L}(D)$. 
We have $m=1$. 
Since a similar argument is applicable for $u$, it follows that
$$ g(t, u)=a t u+b t+c u+d$$
for some $a, b, c, d \in k$. 
Let $P' \in {\rm supp}(D)$ with $\psi(P') \ne (1:0:0)$, $(0:1:0)$.  
If $a \ne 0$, then ${\rm ord}_{P'} g={\rm ord}_{P'}(t u)=-2$, namely, this is a contradiction to $g \in \mathcal{L}(D)$. 
Therefore, $a=0$ and 
$$ g(t, u)=b t+c u+d$$
for some $b, c, d \in k$. 
\end{proof}

\begin{prop} \label{linearity}
There exists an injection 
$$ {\rm Aut}(X_n) \hookrightarrow PGL(3, k). $$
\end{prop} 

\begin{proof}
Let $\sigma \in {\rm Aut}(X_n)$.
By assertion (c) in Theorem \ref{main1}, $\sigma^*D=D$. 
Thus $\sigma$ acts on the complete linear system $|D|$. 
By Proposition \ref{completeness}, $\dim |D|=2$, 
and the  result  follows.   
\end{proof}

It will be shown  that the image ${\rm Lin}(X_n) \subset PGL(3, k)$ of ${\rm Aut}(X_n)$ for the injection described in Proposition \ref{linearity} coincides with the subgroup $\Sigma \subset {\rm Lin}(X_n)$ generated by all automorphisms of type (a), (b), (c) and (d). 
Let $\varphi_0 \in {\rm Lin}(X_n)$.
Considering the plane model $C_n'$,   $\varphi_0$ acts on ${\rm Sing}(C_n')$. 
Since the cyclic subgroup of order $q+1$ described as automorphisms of type (d) acts  transitively on ${\rm Sing}(C_n')$, there exists $\varphi_d \in \Sigma$ of type (d) such that $\varphi_d \varphi_0$ fixes $(1:0:0)$. 
Then $\varphi_d\varphi_0(0:1:0)=(-\delta:1:0)$ for some $\delta$ with $\delta^q+\delta=0$. 
Consider an automorphism $\varphi_b(u, v)=(t+\delta u, u)$ of type (b) so that $\varphi_b\varphi_d\varphi_0$ fixes points $(1:0:0)$ and $(0:1:0)$. 
Further, there exists an automorphism $\varphi_c$ of type (c) such that $\varphi_c\varphi_b\varphi_d\varphi_0$ fixes three points on $\{Z=0\}$, namely, this automorphism fixes all points of $\{Z=0\}$. 
There exists an automorphism $\varphi_a$ of type (a)  such that $\varphi_a\varphi_c\varphi_b\varphi_d\varphi_0$ fixes also $(0:0:1)$. 
It follows that the latter automorphism is of type $(t, u) \mapsto (\lambda t, \lambda u)$ for some $\lambda \in k$. 
On the other hand, the equation  of $C_n'$ implies $\lambda^{q+1}=1$, that is, this automorphism is of type (c), and then $\varphi_0 \in \Sigma$. 
The above argument also implies that the stabilizer subgroup $\Sigma(P)$ of $P=(1:0:0)$ is generated by automorphisms of type (a), (b), and (c). 
Direct inspection shows that $|\Sigma(P)|=q^{4n+1}(q^2-1)$, which implies the former assertion of Theorem \ref{main1} (d).

Since ${\rm Lin}(X_n)$ acts on ${\rm Sing}(C_n')$ on the line $Z=0$, we have a homomorphism  
$$ r: {\rm Lin}(X_n) \rightarrow PGL(\{Z=0\}(\mathbb{F}_{q^2})) \cong PGL(2, \mathbb{F}_{q^2}),  $$
which is given by the restrictions. 
The kernel of $r$ contains $q^{4n}(q+1)$ automorphisms given by
$$ (x, y) \mapsto (\lambda x+\beta, \lambda y+\gamma), $$
where $\lambda^{q+1}=1$ and $L(\beta)=L(\gamma)=0$. 
As we observed above, $|{\rm Lin}(X_n)(P)|=q^{4n+1}(q^2-1)$ for the stabilizer subgroup ${\rm Lin}(X_n)(P)$ of a point $P=(1:0:0)$. 
Considering the orbit-stabilizer theorem and automorphisms of type (a), (b) and (c), we have $|r({\rm Lin}(X_n)(P))|=q(q-1)$ and $|r({\rm Lin}(X_n))|=q(q-1)(q+1)$. 
Let $\delta \in \mathbb{F}_{q^2}$ with $\delta^{q-1}=-1$ and let $V \in PGL(\{Z=0\}(\mathbb{F}_{q^2}))$ with
$$ V(t: u)=(t: \delta u). $$
Then $V({\rm Sing}(C_n'))=\{Z=0\}(\mathbb{F}_q)$, and $V(r({\rm Lin}(X_n)))V^{-1}=PGL(2, \mathbb{F}_q)$. 
The latter assertion of Theorem \ref{main1} (d) follows.

\begin{rem} The following is an explicit list of all automorphisms of the curve $C_n$.
\begin{enumerate}[\rm(i)]
\item $(x,y)\mapsto (\alpha^q x- \lambda \beta ^q y+\delta, \beta   x+  \lambda \alpha y+\epsilon)$,
where $\alpha, \beta  \in \mathbb{F}_{q^2}^*$ are such that   $\alpha^{q+1}+\beta^{q+1}=1$, $\lambda^{q+1}=1$, and  $L(\delta)= L(\epsilon)=0$.
\item $(x,y)\mapsto ( \alpha y+\delta, \beta  x+ \epsilon)$ or  $(x,y)\mapsto (\alpha x+\delta, \beta y +\epsilon)$,
where  $\alpha^{q+1}=\beta^{q+1}=1$, and      $L(\delta)= L(\epsilon)=0$.
\end{enumerate}
In fact, it can be easily checked that  $C_n$ is invariant by each of these automorphisms. 
Note that there are  exactly  $(q^2-q-2)(q+1)^2q^{4n}$ maps of type (i)  and $2(q+1)^2q^{4n}$ maps of   type (ii).
Thus  from  Theorem \ref{main1}, assertion (d),  the list is complete. 
\end{rem}

Let us  consider assertion (e) in Theorem \ref{main1}. 
Let $P=(\beta:1:0) \in \{Z=0\} \setminus {\rm Sing}(C_n)$ with $\beta \in \mathbb{F}_{q^2}$. 
Then $\beta^{q+1} \ne -1$.
Consider the projection $\pi_P$ from $P$, represented by $(x, y) \mapsto x-\beta y$. 
For  $v=x-\beta y$,  
it follows that 
$$ \left\{L(y)+\frac{\beta}{\beta^{q+1}+1}L(v)\right\}^{q+1}+\frac{1}{(\beta^{q+1}+1)^{q+1}}L(v)^{q+1}+\frac{c}{\beta^{q+1}+1}=0. $$ 
Since $\beta \in \mathbb{F}_{q^2}$ and 
$$L\left(y+\frac{\beta}{\beta^{q+1}+1}v\right)=L(y)+\frac{\beta}{\beta^{q+1}+1}L(v), $$ 
 $w=y+\beta v/(\beta^{q+1}+1)$ gives $k(w, v)=k(y, v)$ and 
$$ L(w)^{q+1}+\frac{1}{(\beta^{q+1}+1)^{q+1}}L(v)^{q+1}+\frac{c}{\beta^{q+1}+1}=0.$$
It is inferred that the extension $k(w, v)/k(v)=k(y, v)/k(v)$ is  Galois and of degree $q^{2n}(q+1)$. 
Therefore, there exist $(q^2+1)-(q+1)=q^2-q$ outer Galois points on the line $\{Z=0\}$. 
Conversely,  assume that $R \in \mathbb{P}^2 \setminus C_n$ is an outer Galois point for $C_n$. 
The associated Galois group at $R$ is denoted by $G_R$.   
Using assertion (c) in Theorem \ref{main1},  the following is  obtained.
 
\begin{lem} 
All outer Galois points $R \in \mathbb{P}^2 \setminus C_n$ are contained in the line defined by $Z=0$. 
\end{lem}

\begin{proof}
Assume that $R \not\in \{Z=0\}$. 
Let $P \in {\rm supp}(D)$. 
By assertion (c) in Theorem \ref{main1}, $G_R$ acts on $\varphi^{-1}(\overline{R\varphi(P)}) \cap {\rm supp}(D)=\varphi^{-1}(\varphi(P))$ (see also \cite[Theorem 3.7.1]{stichtenoth}), where $\varphi: X_n \rightarrow C_n$ is the normalization, and $\overline{R\varphi(P)}$ is the line passing through points $R$ and $\varphi(P)$. 
This implies that for the projection $\pi_R$, the fiber of $\pi_R \circ \varphi$ over the point of $\mathbb{P}^1$ corresponding to the line $\overline{R\varphi(P)}$ coincides with $\varphi^{-1}(\varphi(P))$. 
Since $\varphi^{-1}(\varphi(P))$ consists of $q^{2n} <q^{2n} (q+1)=\deg C_n$ points, $\pi_R$ is ramified at each point of $\varphi^{-1}(\varphi(P))$ (see \cite[Corollary 3.7.2]{stichtenoth}). 
However, the directions of the tangent lines at $\varphi(P)$  differ. 
This is a contradiction. 
Therefore, $R \in \{Z=0\}$. 
\end{proof} 

The order of $G_R$ coincides with $\deg C_n=q^{2n}(q+1)$. 
Since $G_R$ can be embedded in $PGL(3, k)$, there exists a cyclic subgroup $C_R \subset G_R$ of order $q+1$ such that the set
$$ F[R]=\{Q \in \mathbb{P}^2 \ | \ \sigma(Q)=Q \ \mbox{ for all } \ \sigma \in C_R\}$$
coincides with $\{R\} \cup \ell_R$ for some line $\ell_R \not\ni R$, and that the set of all fixed points of each element for $C_R \setminus \{1\}$ coincides with $F[R]$ (see \cite[Theorem 2]{fukasawa1}, \cite[Proposition 4]{miura}). 
Since $C_R$ acts on $(\{Z=0\}(\mathbb{F}_{q^2})) \setminus {\rm Sing}(C_n)$, and $|C_R|=q+1$ does not divide  $q^2-q$, it follows that $C_R$ has short orbits. 
Therefore, $C_R$ fixes exactly two points of $(\{Z=0\}(\mathbb{F}_{q^2})) \setminus {\rm Sing}(C_n)$. 
This implies that $R \in (\{Z=0\}(\mathbb{F}_{q^2})) \setminus {\rm Sing}(C_n)$.

Let us  consider assertion (f) in Theorem \ref{main1}. 
Let $G \subset {\rm Lin}(X_n)$ be a subgroup generated by the Galois groups $G_R$ of $q^2-q$ outer Galois points $R \in \mathbb{P}^2 \setminus C_n'$. 
According to automorphisms in the Galois groups of two outer Galois points $(1:0:0)$ and $(0:1:0)$ for the plane model $C_n$, the kernel of $r$ is contained in the group $G$.  
Note that $r(C_R)$ is a cyclic group of order $q+1$ for each outer Galois point $R$, and that $r(C_R)=r(C_{R'})$ if and only if $R' \in F[R]$. 
Therefore, $r(G)$ contains $(q^2-q)/2$ cyclic groups of order $q+1$ and $r(G)$ is generated by such cyclic groups. 
Since $r(C_R)$ acts on ${\rm Sing}(C_n')$ transitively for each outer Galois point $R$, the stabilizer subgroup of $r(G)$ of $P \in {\rm Sing}(C_n')$ contains at least $(q^2-q)/2$ elements. 
Therefore, $|r(G)|=(q^3-q)/2$ or $q^3-q$, according to the order $|r({\rm Lin}(X_n))|$.
Since there exist $(q^2-q)/2$ cyclic groups of order $q+1$ contained in $V r(G) V^{-1}$, it follows that $V r(G) V^{-1} \not\subset PSL(2, \mathbb{F}_q)$, and that $V r(G) V^{-1}=PGL(2, \mathbb{F}_q)$.  
The assertion follows.

\section{Proof of Theorem \ref{main2}} 

A  polynomial  $f(x) \in \mathbb{F}_q[x]$ of degree $d\geq 1$ whose value set $V_f=\{f(\alpha)\, : \alpha  \in \mathbb{F}_q  \}$ has size $\lceil   q/d  \rceil $ is called a {\it minimal value set polynomial} (see \cite{borges, clms, mills}). 
An irreducible plane curve $C \subset \mathbb{P}^2$ is said to be {\it $q$-Frobenius nonclassical}, if the image of any smooth point $P$ of $C$ under the $q$-Frobenius map is contained in the tangent line at $P$ (see \cite{hefez-voloch, stohr-voloch}). 
Results relating such polynomials and Frobenius nonclassical curves  can be found in  \cite{borges}, and they  will be used in the proof of Theorem   \ref{main2}. 

\begin{lem}\label{help}  Let  $ L(x)=\alpha_0 x+\alpha_1 x^{q^2}+\cdots+ \alpha_{n-1}x^{q^{2(n-1)}}+x^{q^{2n}}$,
with  $\alpha_0, \ldots, \alpha_{n-1} \in \overline{\mathbb{F}}_q$ and $\alpha_0\ne 0$.   If   $F=L(x)^{q+1}\in \mathbb{F}_{p^s}[x]$  has value set
$V_F=\{0=\gamma_0, \gamma_1,\ldots,\gamma_r\}$, and  $\prod \limits_{i=0}^{r}(F(x)-\gamma_i)=\theta (x^{p^s}-x)F'(x)$ for some $\theta \in \mathbb{F}_{p^s}^*$, then there exist  positive integers $u$ and $m$, and $\omega_{0}, \omega_{1}, \ldots, \omega_{m-1} \in \mathbb{F}_{p^s}$, $\omega_{0}\neq 0$, such that  $p^s=q^{2(n+um)}$,  $r=(q^{2um}-1)/(q+1)$, and 
 \begin{equation}\label{poly-T}
 x\prod \limits_{i=1}^{r}(x-\gamma_i)=\sum \limits_{i=0}^{m-1}\omega_i x^{qt_i}+ x^{qt_m},
\end{equation}
where  $t_i=\frac{q^{2ui-1}+1}{q+1}, \quad i=0,1,\ldots, m$.
\end{lem}

\begin{proof}  From \cite[Theorem 2.2]{borges}, the condition  $\prod \limits_{i=0}^{r}(F(x)-\gamma_i)=\theta (x^{p^s}-x)F'(x)$  implies that $F=L(x)^{q+1}\in \mathbb{F}_{p^s}[x]$
is a minimal value set polynomial. The same condition also gives
\begin{equation}\label{mvsp-trivial} 
L(x)(L(x)^{q+1}-\gamma_1)\cdots (L(x)^{q+1}- \gamma_r)=\theta\alpha_0(x^{p^s}-x).
\end{equation}
Comparing degrees on \eqref{mvsp-trivial} yields $1+(q+1)r=p^s/q^{2n}$. First, let us suppose that $(p,r)\neq (2,1)$.
Thus we have  $r>1$, and  then \cite[Theorem 2.1]{borges} holds for  $F=L(x)^{q+1}\in \mathbb{F}_{p^s}[x]$. Note that  \eqref{mvsp-trivial}
implies that   $L(x)$ has $q^{2n}$ distinct roots in $\mathbb{F}_{p^s}$, and  that each $L(x)^{q+1}-\gamma_i$  has $q^{2n}(q+1)$ distinct roots in $\mathbb{F}_{p^s}$, $i=1, \dots, r$.
Thus using  (iii) and (b) of  \cite[Theorem 2.1]{borges}   and notation therein, we have  $L_0=L(x)$  and $v=q+1$.  Furthermore,   (a) of  \cite[Theorem 2.1]{borges} 
implies that $q+1$ divides $p^k-1$, and $1+(q+1)r=p^{mk}$. The first assertion gives that $p^k=q^{2u}$ for some integer $u\geq 1$, and from  the second assertion it follows that 
$p^s/q^{2n}=p^{mk}=q^{2mu}$, and then $p^s=q^{2(n+mu)}$ and $r=(q^{2um}-1)/(q+1)$, as claimed.  Finally,  from (c) of  \cite[Theorem 2.1]{borges} ,  we have
\begin{eqnarray*}
T(x)=x\prod \limits_{i=1}^{r}(x-\gamma_i) &=&x\sum_{i=0}^{m} \omega_{i} x^{(p^{k i}-1) / v} \\
                                                                  &=&\sum_{i=0}^{m} \omega_{i} x^{q(q^{2ui-1}+1) / (q+1)} \\
                                                                  &=&\sum \limits_{i=0}^{m-1}\omega_i x^{qt_i}+ x^{qt_m},
\end{eqnarray*}
where  $t_i=\frac{q^{2ui-1}+1}{q+1}, \quad i=0,1,\ldots, m$, and this gives the result when $(p,r)\neq (2,1)$. For $(p,r)=(2,1)$, the relation  $1+(q+1)r=p^s/q^{2n}$
entails  $q=p=2$ and  $s=2(n+1)$. Since  $T(x)=x(x-\gamma_1)$, equation  \eqref{poly-T} holds for $m=u=1$, and the result follows.
\end{proof}

We now proceed to prove the first part of Theorem  \ref{main2}.  Let  
$$ L(x)=\alpha_0 x+\alpha_1 x^{q^2}+\cdots+ \alpha_{n-1}x^{q^{2(n-1)}}+x^{q^{2n}},$$
with  $\alpha_0, \ldots, \alpha_{n-1} \in k$ and $\alpha_0\ne 0$,  be such that the irreducible curve
$$L(x)^{q+1}+L(y)^{q+1}+c=0$$
is $p^s$-Frobenius nonclassical. It follows from \cite[Theorem 3.4 and Corollary 3.5]{borges}  that

  \begin{enumerate}[\rm(i)]
  \item The polynomials  $L(x)^{q+1}, L(x)^{q+1}+c \in \mathbb{F}_{p^s}[x]$ are minimal value set polynomials.
  \item  The values sets $
V=\{L(\alpha)^{q+1}\,: \alpha \in \mathbb{F}_{p^s}\}  \hspace{0.1 cm}\text{ and }\hspace{0.1 cm} W=\{-L(\alpha)^{q+1}-c\,: \alpha \in \mathbb{F}_{p^s}\}
$
are equal.

  \item The  polynomial $T(x)= \prod\limits_{\gamma_i \in V=W}(x-\gamma_i)$ is such that 
  \begin{equation}\label{compos}
T(L(x)^{q+1})=\theta\alpha_0(x^{p^s}-x)L(x)^q, \quad \text{ for some } \theta \in \mathbb{F}_{p^s}^*.
  \end{equation}
  \end{enumerate}
From Lemma \ref{help},  there exist  positive integers $u$ and $m$, and $\omega_{0}, \omega_{1}, \ldots, \omega_{m-1} \in \mathbb{F}_{p^s}$, $\omega_{0}\neq 0$, such that  $p^s=q^{2(n+um)}$, and 
 $$T(x)=x\prod \limits_{i=1}^{r}(x-\gamma_i)=\sum \limits_{i=0}^{m-1}\omega_i x^{qt_i}+ x^{qt_m}
$$
where  $t_i=\frac{q^{2ui-1}+1}{q+1}, \quad i=0,1,\ldots, m$.  Note that (ii) yields $T(-x)=(-1)^{r+1} T(x-c) $, and then the expansion
$$(x-c)^{qt_m}=x^{qt_m}-c^qx^{q(t_m-1)}+\cdots+(-c)^{qt_m}
$$
implies that $t_m=1$, as    $T(x)$ has no monomial of positive degree $q(t_m-1)>qt_i$,   $i\leq m-1.$
Hence  $u=m=1$,  and then  $p^s=q^{2(n+1)}$ and   $T(x)=x^q+\omega_0 x$.
Note that the condition  $-c \in V$  given by (ii) implies that $c \in V$, as $T(c)=-T(-c)=0$. Let $\alpha \in \mathbb{F}_{q^{2(n+1)}}^*$
be such that   $L(\alpha)^{q+1}=c$, and let us  define  $\beta:=1/L(\alpha)$. It is clear that  $T(1/\beta^{q+1})=0$, and then  $T(x)=x^q-1/\beta^{q^2-1}x$.   Therefore,  since $L$ and $T$  are monic, equation \eqref{compos} yields
 \begin{equation}\label{poly-L}
L(x)^{q^2}-1/\beta^{q^2-1}L(x)=x^{q^{2(n+1)}}-x.
\end{equation}
That is, $L(x)$ is a root of $P(T)=T^{q^2}-1/\beta^{q^2-1}T-(x^{q^{2(n+1)}}-x) \in \mathbb{F}_{q^{2(n+1)}}[x][T]$. One can  check that
$\big\{ \sum\limits_{i=0}^{n}\beta^{q^{2(i+1)}-1} x^{q^{2i}}+\lambda/\beta\, : \lambda \in \mathbb{F}_{q^2}\big\}$ is the set of  all  roots of $P(T)$, and since 
$L(0)=0$, it follows that   $L(x)=\sum\limits_{i=0}^{n}\beta^{q^{2(i+1)}-1} x^{q^{2i}}$.   Note that
$$ \beta L(x/\beta^{q^2})= x+x^{q^2}+\cdots+x^{q^{2n}},$$
and $\beta^{q+1}c =1$  imply that $C_n$ is $\mathbb{F}_{q^{2(n+1)}}$-projectively equivalent to the curve of affine equation
$$(x+x^{q^2}+\cdots+x^{q^{2n}})^{q+1}+(y+y^{q^2}+\cdots+y^{q^{2n}})^{q+1}+1=0.$$
Conversely, if  $L(x)=x+x^{q^2}+\cdots+x^{q^{2n}}$, then clearly $L(x)^{q+1}$ and $-L(y)^{q+1}-1$  are minimal value set polynomials  with the same value set    $V=\mathbb{F}_{q}\subset \mathbb{F}_{q^{2(n+1)}}$. Therefore, from  \cite[Corollary 3.5]{borges}, the curve   $L(x)^{q+1}+ L(y)^{q+1}+1=0$ is $q^{2(n+1)}$-Frobenius nonclassical.

To prove assertion (a),  let us recall that $C_n$ is a curve of degree $d=q^{2n}(q+1)$ with $q+1$ ordinary sigularities $P_i \in C_n(\mathbb{F}_{q^{2(n+1)}})$ of multiplicity $m_{P_i}=q^{2n}$ (Remark \ref{rem2}).
Therefore, \cite[Corollary 3.1]{borges-homma} gives
\begin{eqnarray*}
\#X_n(\mathbb{F}_{q^{2(n+1)}})&=&d(q^{2(n+1)}-d+2)+\sum_{i=1}^{q+1} m_{P_i}\left(m_{P_i}-1\right) \\
                                                                  &=&q^{2n}(q+1)(q^{2(n+1)}-q^{2n}(q+1)+2)+(q+1) q^{2n}\left(q^{2n}-1\right)  \\
                                                                  &=&q^{4n+3}-q^{4n+1}+q^{2n+1}+q^{2n}.
\end{eqnarray*}

To prove assertion (b),  let us recall that  that  each of the   $q+1$ ordinary sigularities $P_i \in C_n(\mathbb{F}_{q^{2(n+1)}})$ has all of its $q^{2n}$ tangent lines defined over $\mathbb{F}_{q^{2(n+1)}}$ (Remark \ref{rem2}).
This implies that each such singularity gives rise to $q^{2n}$ points of  $X_n(\mathbb{F}_{q^{2(n+1)}})$. Therefore,
\begin{eqnarray*}
\#C_n(\mathbb{F}_{q^{2(n+1)}})&=&\#X_n(\mathbb{F}_{q^{2(n+1)}}) -(q+1)(q^{2n} -1)\\
                                                                  &=&q^{4n+3}-q^{4n+1}+q+1.
\end{eqnarray*}

We proceed to prove assertion (c). See \cite{gptu} for the definition of a $(k, d)$-arc (arising from a plane curve) and the arc property. 
Note that there exists an $\mathbb{F}_{q^{2(n+1)}}$-line containing an outer Galois point $(1:0:0)$, which is not a tangent line. 
Therefore, our arc forms a $(q^{4n+3}-q^{4n+1}+q+1, q^{2n+1}+q^{2n})$-arc (see also \cite[Lemma 2.2]{gptu}).  
For the incompleteness of our arc, we prove the following.

\begin{prop} 
Let $ \lambda\neq 0$ be a root of $L(x)$, and let $T_\lambda:=\{\alpha/\lambda: L(\alpha)=0\}$.  For $a \in \mathbb{F}_{q^{2(n+1)}} \setminus T_\lambda$ and $R=(a:1:0)$,
the  $\mathbb{F}_{q^{2(n+1)}}$-lines $\ell \ne \{Z=0\}$ passing through $R$ intersect $ C_n(\mathbb{F}_{q^{2(n+1)}})$ in less than $q^{2n}(q+1)$ points.
\end{prop} 

\begin{proof} 
For $a \in \mathbb{F}_{q^{2(n+1)}} \setminus T_\lambda$, let us consider the lines $\ell_m: x=a y+m$, $m \in \mathbb{F}_{q^{2(n+1)}}$, passing through $R=(a:1:0)$. 
Suppose the number of points of $(\ell_m \cap C_n)(\mathbb{F}_{q^{2(n+1)}})$ is equal to $q^{2n}(q+1)$ for some $m \in \mathbb{F}_{q^{2(n+1)}}$. Thus, for 
$L_\lambda(x)=x^{q^2}-\lambda^{q^2-1}x$,  the $q^{2(n+1)}$ elements of 
$$ \{ \sigma_{\beta, \gamma}: (x, y) \mapsto (x+\beta, y+\gamma) \ | \ L(\beta)=L_\lambda(\gamma)=0 \} $$
give rise to a pencil of $q^{2(n+1)}$ distinct lines $\ell_{m+\beta-a\gamma}\ni R$ defined over $\mathbb{F}_{q^{2(n+1)}}$ intersecting $C_n$ in $q^{2(n+1)}\cdot q^{2n}(q+1)$ points in $\mathbb{A}^2(\mathbb{F}_{q^{2(n+1)}})$, as $\ell_{m+\beta-a\gamma}=\ell_m$ if and only if $\beta=a\gamma$, that is, $a \in T_\lambda$. 
Therefore, the set $C_n(\mathbb{F}_{q^{2(n+1)}})$ consists of at least $(q+1)(q^{4n+2}+1)$ points. 
This is a contradiction to assertion (c). 
\end{proof}

Assertion (d) in Theorem \ref{main2} is verified.

\begin{rem}
The irreducibility of the curve defined by 
$$ (x^{q^n}+x^{q^{n-1}}+\cdots+x^q+x)^\frac{q-1}{q'-1}+(y^{q^n}+y^{q^{n-1}}+\cdots+y^q+y)^\frac{q-1}{q'-1}+c=0 $$
and its genus are determined in \cite{fukasawa2}, where $q$ and $q'$ are powers of $p$, $q'-1$ divides $q-1$, and $c \in \mathbb{F}_{q'}\setminus \{0\}$. 
Similarly to assertion (a) in Theorem \ref{main2}, it can be proved that this curve is $q^{n+1}$-Frobenius nonclassical.  
\end{rem}

\begin{center}
{\bf Acknowledgments} 
\end{center}
The second author is grateful to Doctor Kazuki Higashine for helpful discussions.

\end{document}